\documentclass[12pt]{amsart}
\usepackage{amssymb,latexsym,amsmath,amscd,amsthm,graphicx, color}
\usepackage[all]{xy}
\usepackage{pgf,tikz}
\usepackage{mathrsfs}
\usepackage{graphicx}
\usepackage{epstopdf}
\usepackage{breqn}
\usepackage[numbers,sort&compress]{natbib}
\usepackage{amsmath}
\usepackage{xcolor}
\usepackage{fancyhdr}
\usetikzlibrary{arrows}
\newcommand{\quotes}[1]{``#1''}
\definecolor{uuuuuu}{rgb}{0.26666666666666666,0.26666666666666666,0.26666666666666666}
\definecolor{xdxdff}{rgb}{0.49019607843137253,0.49019607843137253,1.}
\definecolor{ffqqqq}{rgb}{1.,0.,0.}

\newtheorem{theorem}{Theorem}[section]
\newtheorem{lemma}{Lemma}[section]

\newtheorem{remark}{Remark}[section]
\newtheorem{definition}[theorem]{Definition}

\newtheorem{rule-def}[theorem]{Rule}
\numberwithin{equation}{section}

\numberwithin{equation}{section}
\usepackage{hyperref}
\usepackage{xcolor}
\hypersetup{
  colorlinks   = true,
  linkcolor    = blue,
  citecolor   = blue,
  urlcolor   = blue,
  pdftitle={SSS }
}

\begin{document}
\title[]{Boundedness of $p$-adic Hardy-Hilbert type integral operator on Block spaces }

\author{Salman Ashraf}
\address{School of Mathematical and Statistical Sciences, Indian Institute of Technology Mandi, Kamand (H.P.) - 175005, India}
\email{ashrafsalman869@gmail.com}

\keywords{$p$-adic field, $p$-adic Hardy-Hilbert type integral operator, Block space, Morrey space.}

\begin{abstract}
In this paper, we estimate an operator norm of dilation operators on block spaces ($\mathfrak{B}_{r,\alpha}(\mathbb{Q}_p)$) over $p$-adic field. With this estimate, we establish boundedness of $p$-adic Hardy-Hilbert type integral operator on $\mathfrak{B}_{r,\alpha}(\mathbb{Q}_p)$. Moreover as application to our result, we obtain the $p$-adic Hilbert inequality, $p$-adic Hardy inequality and $p$-adic Hardy-Littlewood-P\'olya inequality on $\mathfrak{B}_{r,\alpha}(\mathbb{Q}_p)$. 
\end{abstract}

\maketitle

\section{Introduction}

Classical block space is a generalization of Lebesgue space. Zorko \cite{Z} introduced the block spaces, and proved that block space is the predual of the classical Morrey spaces. Since the blocks considered by Zorko \cite{Z} has mean zero, therefore later in \cite{1}, the authors prove that mean zero condition on blocks can be omitted and still have a description of the block spaces as a predual of Morrey spaces. The reader is referred to \cite{AA1,YN}, for some recent developments of Morrey spaces and their related function spaces on $\mathbb{R}^n.$

The Hardy-Littlewood-P\'olya inequality for Lebesgue space $L^q(\mathbb{R}_+)$ were established in \cite{HA}, which unify several important results in analysis such as the Hardy inequality and the Hilbert inequality. Also, the Hardy-Littlewood-P\'olya inequality for block spaces on $\mathbb{R}_+$ were established in \cite{HK3}. One can refer to \cite{ED,KR,YA,YA1} for detail study of Hardy inequality and Hilbert inequality and it's related topics. 

The main aim of this paper is to establish boundedness of $p$-adic Hardy-Hilbert type integral operator on block spaces over $p$-adic field. As consequences, we establish $p$-adic Hardy inequality, $p$-adic Hilbert inequality and $p$-adic Hardy-Littlewood-P\'olya inequality for block spaces over $p$-adic field.  In \cite{3}, K. P. Ho introduced and discussed some fundamental properties of block spaces over locally compact Vilenkin group. He also obtained the boundedness of the Hardy-Littlewood maximal function on block spaces over locally compact Vilenkin group. Since additive group of $p$-adic field is an example of locally compact Vilenkin group, therefore in section 2, we define the block spaces over $p$-adic field by the idea in \cite{3,L1}.

In 2020, Huabing Li and Jianjun Jin \cite{4} introduced and studied the $p$-adic Hardy-Hilbert type integral operator. It should be pointed out that the $p$-adic Hardy-Hilbert type integral operator includes many classical operators in $p$-adic Harmonic analysis such as the $p$-adic Hardy operator, $p$-adic Hilbert operator and  $p$-adic Hardy-Littlewood-P\'olya operator. Let us now recall the definition of $p$-adic Hardy-Hilbert type integral operator.

Let $f$ be a nonnegative integrable function on $\mathbb{Q}_p,$ and $\mathcal{K} : \mathbb{R}_+ \times \mathbb{R}_+ \rightarrow [0, \infty)$ be a homogeneous function of degree $-1,$ that is, $\mathcal{K}(\xi x,\xi y) = \xi^{-1} \mathcal{K}(x,y)$ for any $\xi >0.$ Then $p$-adic Hardy-Hilbert type integral operator with the kernel $\mathcal{K}$ is defined by 
\begin{align} \label{eq3}
\mathscr{T}^pf(x) = \int\limits_{\mathbb{Q}^*_p} \mathcal{K}(|x|_p,|y|_p)f(y)dy,~~~~~ x \in \mathbb{Q}^*_p. 
\end{align} 
Let us now see that for some special case of the kernel $\mathcal{K},~\mathscr{T}^p$ reduces to some important operators
in $p$-adic analysis.
\begin{itemize}
\item[(\romannumeral 1)] Let us choose 
\begin{align*}
\mathcal{K}(|x|_p,|y|_p) = \dfrac{1}{|x|_p~+~|y|_p},
\end{align*}
then we have $p$-adic Hilbert operator defined by 
\begin{align} \label{Hilbert}
H^pf(x) = \int\limits_{\mathbb{Q}^*_p} \dfrac{f(y)}{|x|_p~+~|y|_p}dy,~~~~~~~~~ x \in \mathbb{Q}^*_p.
\end{align}
\item[(\romannumeral 2)] Let us choose 
\begin{align*}
\mathcal{K}(|x|_p,|y|_p) = |x|_p^{-1} \Phi_E(|y|_p),
\end{align*}
where $\Phi_E$ is the characteristic function of $E = \{ y \in \mathbb{Q}_p : |y|_p \leq |x|_p\},$ then we have $p$-adic Hardy operator
\begin{align} \label{Hardy}
\mathcal{H}^pf(x) = \dfrac{1}{|x|_p}~\int\limits_{|y|_p \leq |x|_p} f(y)dy,~~~~~~~~~~~ x \in \mathbb{Q}^*_p.
\end{align} 
\item[(\romannumeral 3)] By choosing 
\begin{align*}
\mathcal{K}(|x|_p,|y|_p) = \dfrac{(|x|_p|y|_p)^{\frac{\lambda}{2}}}{\text{max}\{|x|_p,|y|_p\}^{\lambda+1}},~~~~\lambda \geq 0,
\end{align*}
we obtain the operator
\begin{align} \label{polya}
\mathscr{D}^pf(x) = \int\limits_{\mathbb{Q}^*_p}\dfrac{(|x|_p|y|_p)^{\frac{\lambda}{2}}}{\text{max}\{|x|_p,|y|_p\}^{\lambda+1}} f(y)dy,~~~~~ ~~~~~~x \in \mathbb{Q}^*_p. 
\end{align} 
Observe that for $\lambda =0,~\mathscr{D}^p$ reduces to $p$-adic Hardy-Littlewood-P\'{o}lya operator defined by
\begin{align*}
\mathscr{P}^pf(x) = \int\limits_{\mathbb{Q}^*_p} \dfrac{f(y)}{\text{max}\{|x|_p,|y|_p\}}dy,~~~~~~~~~ x \in \mathbb{Q}^*_p. \nonumber
\end{align*} 
\end{itemize}

In \cite{5}, K. H. Dung and D. V. Duong  gave the necessary and sufficient condition for the boundedness of $p$-adic Hardy-Hilbert type integral operator on two weighted Morrey spaces and Morrey-Herz spaces. Also, in \cite{9}, they established the boundedness of $\mathscr{T}^p$ on weighted Triebel–Lizorkin space. Recently, Batbold, Sawano and Tumendemberel \cite{10} introduced $m$-linear $p$-adic integral operator which is similar to certain multilinear integral operators on Euclidean spaces (see \cite{11}) . More precisely, Let  $K : \mathbb{R}^m_+  \rightarrow [0, \infty)$ be a homogeneous function of degree $-m,$ that is, $K(\xi x_1,\xi x_2,\cdots,\xi x_{m+1}) = \xi^{-m} K( x_1, x_2,\cdots, x_{m+1})$ for any $\xi >0.$ Then $m$-linear $p$-adic integral operator with the kernel $K$ is defined by
\begin{align*}
T_m^p(f_1,f_2,\cdots,f_m)(x) = \int\limits_{(\mathbb{Q}^*_p)^m} K(|x|_p,|y_1|_p,\cdots,|y_m|_p)\prod\limits_{i=1}^{m}f_i(y_i)dy_1dy_2\cdots dy_m,~~~~~ x \in \mathbb{Q}^*_p. 
\end{align*} 

Batbold, Sawano and Tumendemberel \cite{10} gave necessary and sufficient conditions for the boundedness of $m$-linear $p$-adic integral operator on $p$-adic Lebesgue spaces and Morrey spaces with power weights. In \cite{7}, Duong and Hong obtain the boundedness of $m$-linear $p$-adic integral operator on two weighted Herz spaces. As an application to their result, they obtain the boundedness of $p$-adic multilinear Hilbert operator, $p$-adic multilinear Hardy operator and $p$-adic multilinear Hardy–Littlewood–P\'olya operator on two-weighted Herz spaces.

A local field is a locally compact, totally disconnected, non-Archimedian norm valued and non-discrete topological field, see \cite{book_impulse6} for
basic Fourier analysis on local fields. The basic archetypes of local fields are the $p$-adic field $\mathbb{Q}_p$, and a field of formal Laurent series
$\mathbb{F}_q((t))$ over the finite field with $q$ elements. In recent years, the study of Harmonic and Wavelet analysis on local fields has received a lot of attention (see \cite{bib1,BJ1,BJ2,CH4,CH5,R} and references therein).

The study of operators on local fields, is quite new and lots of new topics are worth to explore. For further information on boundedness of some fundamental operators in Harmonic analysis, like maximal operator, singular integral operator, dilation operators, Hardy operator, Hardy–Ces\'aro operator and Hausdorff operator on function spaces over local fields (see \cite{QJ1,QJ2,CH,CH2,CH3,5,9,MB,SI,FU,new4,new8,new3}).

This paper is organized as follows. In section 2, we provide a brief introduction to $p$-adic analysis as well as definition of block spaces over $p$-adic field. In section 3, we estimate an operator norm of dilation operator on block spaces over $p$-adic field. With this estimate, we establish boundedness of $p$-adic Hardy-Hilbert type integral operator on block spaces over $p$-adic field. Finally as an application of the boundedness of $p$-adic Hardy-Hilbert type integral operator, we obtain the $p$-adic Hilbert inequality, $p$-adic Hardy inequality and $p$-adic Hardy-Littlewood-P\'olya inequality for block spaces over $p$-adic field.

\section{Preliminaries}
\subsection{The field of $p$-adic numbers ($\mathbb{Q}_p$)} Let $p$ be any fixed prime in $\mathbb{Z}.$ Define the $p$-adic absolute value (or $p$-adic norm) $|\cdot|_p$ on $\mathbb{Q}$ by  
\begin{align*}
|x|_p = \begin{cases} p^{-\gamma}~ & \text{if}~~ x = p^{\gamma} \frac{m}{n} \\
0~ & \text{if}~~ x = 0, \\ \end{cases}
\end{align*}

where $\gamma,m,n \in \mathbb{Z}$ and $m,n$ are not divisible by $p.$ The field of $p$-adic numbers, denote by $\mathbb{Q}_p,$ is the completion of the field of rational numbers $\mathbb{Q}$ with respect to the metric $d_p(x,y) = |x-y|_p.$ It is easy to see that $p$-adic absolute value satisfy the following properties: 

\begin{itemize}
\item[(a)] $|xy|_p= |x|_p|y|_p$ for all $x,y \in \mathbb{Q}_p$;
\item[(b)] $|x+y|_p \leq \text{max} \{|x|_p, |y|_p\},$  for all $x,y \in \mathbb{Q}_p$,
\end{itemize}

The property (b) is called the \textit{ultrametric inequality(or the non-Archimedean property)}. It follows that
\begin{equation*}
|x+y|_p = \text{max} \{|x|_p, |y|_p\}~ \text{if}~ |x|_p \neq |y|_p.\\
\end{equation*} 

$\mathbb{Q}_p$ with natural operations and topology induced by the
metric $d_p$ is a locally compact, non-discrete, complete and totally disconnected field. It is also well known that any non-zero $p$-adic number $x \in \mathbb{Q}_p$ can be uniquely represented in the canonical series
\begin{align} \label{s1}
x =p^\gamma\sum\limits_{l=0}^{\infty}c_lp^l,
\end{align}

where $c_l \in \mathbb{Z}/p\mathbb{Z}$ and $c_0 \neq 0.$ The series (\ref{s1}) is convergence in $p$-adic norm since $|c_lp^l|_p \leq p^{-l}.$ For $ a \in \mathbb{Q}_p$ and $k \in \mathbb{Z},$ we denote by 
\begin{align*}
B^k(a) = \{ x \in \mathbb{Q}_p : |x-a|_p \leq p^k\}, \\
S^k(a) = \{ x \in \mathbb{Q}_p : |x-a|_p = p^k\},
\end{align*}
respectively, a ball and a sphere of radius $p^k$ and center at $a.$ We use the notations $B^k =B^k(0)$ and $S^k=S^k(0).$  The set $\{B^k \subset \mathbb{Q}_p: k \in \mathbb{Z}\}$ satisfies the following:

\begin{itemize}
\item[(\romannumeral 1)] $\{B^k \subset \mathbb{Q}_p : k \in \mathbb{Z}\}$ is a base for neighborhood system of identity in $\mathbb{Q}_p$, and $B^{k} \subset B^{k+1},~ k \in \mathbb{Z};$
\item[(\romannumeral 2)] $B^k ,~ k \in \mathbb{Z},$ is open, closed and compact in $\mathbb{Q}_p$;
\item[(\romannumeral 3)] $\mathbb{Q}_p = \bigcup\limits_{{k=-\infty}}^{+\infty} B^k~~\text{and}~~\{0\}=\bigcap\limits_{{k=-\infty}}^{+\infty} B^k.$
\end{itemize}  
We also have, 
\begin{align*}
\mathbb{Q}^*_p = \mathbb{Q}_p \setminus  \{0\} = \bigcup\limits_{{k=-\infty}}^{+\infty} S^k.
\end{align*}

Since additive group of $\mathbb{Q}_p$ is a locally compact Abelian group, we choose a Haar measure $dx$ on additive group of $\mathbb{Q}_p$ normalized so that

\begin{align*}
|B^0| = \int_{B^0} dx =1,
\end{align*} 

where $|E|$ denotes the Haar measure of a measurable set $E$ of $\mathbb{Q}_p.$ Then by a simple calculation, the Haar measures of any balls and spheres can be obtained. Especially, we frequently use

\begin{align*}
|B^k| = p^{k}~~~~\text{and}~~~~|S^k| = p^k(1-p^{-1}).\\
\end{align*}

\begin{definition}
Let $S(\mathbb{Q}_p)$ be the space of complex-valued, locally constant and compactly supported functions on $\mathbb{Q}_p.$ The space $S(\mathbb{Q}_p)$ is called the \textit{Schwartz space} over $\mathbb{Q}_p,$ has the following properties:
\begin{itemize}
\item[(1)] If $\phi \in S(\mathbb{Q}_p),$ then $\phi$ is continuous.
\item[(2)] If $\phi \in S(\mathbb{Q}_p),$ then there exists $k \in \mathbb{Z}$ such that $\text{supp}~ \phi \subset B^k.$
\item[(3)] If $\phi \in S(\mathbb{Q}_p),$ then there exists $l \in \mathbb{Z}$ such that $\phi$ is constant on the cosets of $B^l.$
\item[(4)] $S(\mathbb{Q}_p)$ is dense in $L^r(\mathbb{Q}_p),~1 \leq r < \infty.$
\end{itemize}
\end{definition} 

We refer to \cite{book_impulse6,6} for details of $p$-adic field and proof of statements discussed in this subsection.

\subsection{Block spaces} We begin with definition of $p$-adic central Morrey spaces (see \cite{CH6,NM}). 

\begin{definition}
Let $\alpha$ be a non-negative real number and $1 \leq r < \infty.$ Then the $p$-adic central Morrey space is defined by
\begin{align*}
M_{r,\alpha}(\mathbb{Q}_p)=\{f \in L_{\text{loc}}^r(\mathbb{Q}_p)~ : ~ \|f\|_{M_{r,\alpha}(\mathbb{Q}_p)} < \infty \}, 
\end{align*}
where 
\begin{align*}
\|f\|_{M_{r,\alpha}(\mathbb{Q}_p)} = \Bigg(\sup_{k \in \mathbb{Z}}~ \dfrac{1}{|B^k|^{\alpha r}} \int_{B^k}|f(x)|^r dx\Bigg)^{1/r}. 
\end{align*}
\end{definition}

It is clear that $M_{r, 0}(\mathbb{Q}_p) = L^r(\mathbb{Q}_p).$ 

\begin{definition} Let $\alpha \in \mathbb{R}$  and $0 < r < \infty.$ A function $a: \mathbb{Q}_p \rightarrow \mathbb{C}$ is said to be a central $(r, \alpha)$- block if there exist $n \in \mathbb{Z}$  such that $\text{supp}~ a \subset B^n$ and satisfies 

\begin{align*}
\|a\|_{L^r(\mathbb{Q}_p)}~ \leq~ |B^n|^{- \alpha}. \\
\end{align*} 

The block space $\mathfrak{B}_{r,\alpha}(\mathbb{Q}_p)$ generated by central $(r, \alpha)$-blocks are defined by 

\begin{align*}
\mathfrak{B}_{r,\alpha}(\mathbb{Q}_p)=\{f \in L_{\text{loc}}^r(\mathbb{Q}_p) : f = \sum_{k=1}^{\infty} \lambda_k a_k,~\sum_{k=1}^{\infty} |\lambda_k| < \infty~\text{and each}~a_k~\text{is a central}~(r, \alpha)\text{-block}\}.
\end{align*}

The block space $\mathfrak{B}_{r,\alpha}(\mathbb{Q}_p)$  is endowed with the norm 

\begin{align*}
\|f\|_{\mathfrak{B}_{r,\alpha}(\mathbb{Q}_p)} = \inf \bigg\{\sum_{k=1}^{\infty} |\lambda_k| ~: ~ f = \sum_{k=1}^{\infty} \lambda_k a_k \bigg\}, \\
\end{align*}

where the infimum is taken over all such decompositions of $f.$
\end{definition}

According to the definition of $\mathfrak{B}_{r,\alpha}(\mathbb{Q}_p),$ for any $\text{central}~(r, \alpha)\text{-block}~b,$ we have
\begin{align*}
\|b\|_{\mathfrak{B}_{r,\alpha}(\mathbb{Q}_p)} \leq 1.
\end{align*}

For $r \in  (1, \infty),$ let $r^\prime$ be the conjugate of $r.$ We have H\"{o}lder's inequality for $M_{r,\alpha}(\mathbb{Q}_p)$ and $\mathfrak{B}_{r,\alpha}(\mathbb{Q}_p)$ \citep[Lemma 3.2]{4}.

\begin{lemma}
Let $1 < r < \infty$ and $\alpha > 0.$ If $f \in M_{r,\alpha}(\mathbb{Q}_p)$ and $g \in \mathfrak{B}_{r^\prime,\alpha}(\mathbb{Q}_p)$ then 

\begin{align*}
\int_{\mathbb{Q}_p} |f(x)g(x)|dx \leq  \|f \|_{M_{r,\alpha}(\mathbb{Q}_p)} \|g \|_{ \mathfrak{B}_{r^\prime,\alpha}(\mathbb{Q}_p)}  .
\end{align*}
\end{lemma}

\subsection{Subspace of Morrey spaces}

Zorko \cite{Z} proved that set of continuous functions with compact support ($C_0(\mathbb{R}^n)$) is not dense in Morrey spaces ($L^{p,\lambda}(\mathbb{R}^n)$) and she introduced an important subset of $L^{p,\lambda}(\mathbb{R}^n)$ so-called \textit{Zorko subspace} $L_0^{p,\lambda}(\mathbb{R}^n),$ which is defined as the closure of $C_0(\mathbb{R}^n)$ in the $L^{p,\lambda}(\mathbb{R}^n)$ norm. Adams and Xiao \cite{AA} stated that $L_0^{p,\lambda}(\mathbb{R}^n)$ is the predual of block spaces ($H^{p^\prime,\lambda}(\mathbb{R}^n)$) and the three spaces $L_0^{p,\lambda}(\mathbb{R}^n)-H^{p^\prime,\lambda}(\mathbb{R}^n)-L^{p,\lambda}(\mathbb{R}^n)$ are similar to $VMO(\mathbb{R}^n)-H^1(\mathbb{R}^n)-BMO(\mathbb{R}^n)$(see \cite{AA1}). In \cite{IE}, Izumi, Sato and Yabuta considered Morrey spaces on unit circle \textbf{T} and proved in detail that $L_0^{p,\lambda}(\textbf{T})$ is the predual of block spaces.

Motivated by above work, for $1 < r < \infty$ and $\alpha > 0,$ we consider the function 
\begin{align*}
f(x) := \begin{cases}
 |x|_p^{\frac{\alpha - 1}{r}} , & ~~~ |x|_p ~\leq 1, \\
  0 , &~~~ |x|_p~ > 1. 
\end{cases}
\end{align*}

Then $ f \in M_{r,\alpha}(\mathbb{Q}_p)$ and for any $g \in S(\mathbb{Q}_p),$ we get $c > 0$ such that
\begin{align*}
\Bigg(\sup_{k \in \mathbb{Z}}~ \dfrac{1}{|B^k|^{\alpha r}} \int_{B^k}|f(x)- g(x)|^r dx\Bigg)^{1/r} \geq c >0.
\end{align*}

Hence, functions in $M_{r,\alpha}(\mathbb{Q}_p)$ cannot be approximated by functions in $S(\mathbb{Q}_p).$ In particular, $S(\mathbb{Q}_p)$ is not dense in $M_{r,\alpha}(\mathbb{Q}_p),$ and therefore we define $\widetilde{M}_{r,\alpha}(\mathbb{Q}_p)$ as the closure of $S(\mathbb{Q}_p)$ in $M_{r,\alpha}(\mathbb{Q}_p).$

Analogous to the classical case, one could expect the following:
\begin{equation} \label{clo}
\widetilde{M}_{r,\alpha}(\mathbb{Q}_p) ~ \overset{\ast}{\longrightarrow} ~ \mathfrak{B}_{r^\prime,\alpha}(\mathbb{Q}_p) \overset{\ast}{\longrightarrow}  ~ M_{r,\alpha}(\mathbb{Q}_p),
\end{equation}

and the spaces $\widetilde{M}_{r,\alpha}(\mathbb{Q}_p) - \mathfrak{B}_{r^\prime,\alpha}(\mathbb{Q}_p) - M_{r,\alpha}(\mathbb{Q}_p)$ have a relationship alike to  $VMO(\mathbb{Q}_p) - H^1(\mathbb{Q}_p)  -BMO(\mathbb{Q}_p).$ The reader might find the papers \cite{MO,BMO,Kim}, where the spaces $ H^1(\mathbb{Q}_p), ~BMO(\mathbb{Q}_p)$ and $VMO(\mathbb{Q}_p)$ are studied and it is also pointed out that $BMO(\mathbb{Q}_p)$ can be characterize as the dual of the space $H^1(\mathbb{Q}_p).$

The following theorem establish Minkowski's integral inequality for block spaces $\mathfrak{B}_{r,\alpha}(\mathbb{Q}_p).$

\begin{theorem} \label{Min} Let $1 < r < \infty$ and $ \alpha > 0.$ Let $F$ be a function on the product space $\mathbb{Q}_p \times \mathbb{Q}_p$ such that $\|F(\cdot, y)\|_{\mathfrak{B}_{r,\alpha}(\mathbb{Q}_p)} \in L^1(\mathbb{Q}_p)$ then we have
\begin{align} \label{kow}
\bigg\|\int\limits_{\mathbb{Q}_p}F(\cdot, y)dy\bigg\|_{\mathfrak{B}_{r,\alpha}(\mathbb{Q}_p)}  \leq \int\limits_{\mathbb{Q}_p} \|F(\cdot, y)\|_{\mathfrak{B}_{r,\alpha}(\mathbb{Q}_p)}dy.
\end{align}
\end{theorem}
\begin{proof}
Let $r^\prime$ be the conjugate of $r$ and  suppose $ g \in \widetilde{M}_{r^\prime,\alpha}(\mathbb{Q}_p)$ with $\|g\|_{\widetilde{M}_{r^\prime,\alpha}(\mathbb{Q}_p)} \leq 1$. Write  
\begin{align*}
G(x) = \int\limits_{\mathbb{Q}_p} F(x,y) dy. 
\end{align*}
By using H\"older's inequality for $M_{r^\prime,\alpha}(\mathbb{Q}_p)$ and $\mathfrak{B}_{r,\alpha}(\mathbb{Q}_p))$ we have that
\begin{align*}
\bigg|\int_{\mathbb{Q}_p} G(x)g(x)dx \bigg| & \leq \int\limits_{\mathbb{Q}_p} \int\limits_{\mathbb{Q}_p} |F(x,y)
| |g(x)| dx dy \\
& \leq \int\limits_{\mathbb{Q}_p} \|F(x, y)\|_{\mathfrak{B}_{r,\alpha}(\mathbb{Q}_p)} \|g\|_{M_{r^\prime,\alpha}(\mathbb{Q}_p)} dy \\ 
& = \int\limits_{\mathbb{Q}_p} \|F(x, y)\|_{\mathfrak{B}_{r,\alpha}(\mathbb{Q}_p)} \|g\|_{\widetilde{M}_{r^\prime,\alpha}(\mathbb{Q}_p)} dy \\
&  \leq \int\limits_{\mathbb{Q}_p} \|F(x, y)\|_{\mathfrak{B}_{r,\alpha}(\mathbb{Q}_p)} dy.
\end{align*}
Taking now on the left the supremum over $ g \in \widetilde{M}_{r^\prime,\alpha}(\mathbb{Q}_p)$ with $\|g\|_{\widetilde{M}_{r^\prime,\alpha}(\mathbb{Q}_p)} \leq 1,$ we obtain that $G \in (\widetilde{M}_{r^\prime,\alpha}(\mathbb{Q}_p))^*.$ Now by (\ref{clo}) we have $G \in \mathfrak{B}_{r,\alpha}(\mathbb{Q}_p)$ and \eqref{kow} is valid.
\end{proof}

\section{Main Results}
First, we study the dilation operators on block spaces. Let $\tau (\neq 0) \in \mathbb{Q}_p$ and for any function $f$ on $\mathbb{Q}_p,$ consider the dilation operator of the form 
\begin{equation*} \label{eq1}
(\mathcal{D}_\tau f)(x) = f(\tau x),\qquad x \in \mathbb{Q}_p.
\end{equation*}

The following theorem, which discovers the boundedness of dilation operator on $\mathfrak{B}_{r,\alpha}(\mathbb{Q}_p),$ is needed in order to prove the main result of this paper.

\begin{theorem} \label{dil}
Let $1 < r < \infty$ and $ \alpha > 0.$ Then, for all $f \in \mathfrak{B}_{r,\alpha}(\mathbb{Q}_p),$ we have
\begin{align*} 
\|\mathcal{D}_\tau f\|_{\mathfrak{B}_{r,\alpha}(\mathbb{Q}_p)} \leq 2|\tau|_p^{-(1/r+\alpha)}  \|f\|_{\mathfrak{B}_{r,\alpha}(\mathbb{Q}_p)}.
\end{align*}
\end{theorem}

\begin{proof}
Let $f \in \mathfrak{B}_{r,\alpha}(\mathbb{Q}_p),$ then by the definition of $\mathfrak{B}_{r,\alpha}(\mathbb{Q}_p)$ there exist a sequence of scalers $\{\lambda_k\}_{k \in \mathbb{N}}$ and family of central $(r, \alpha)$- blocks with the support $B^n$ for $n \in \mathbb{Z}$ such that 

\begin{equation} \label{eq2}
f = \sum_{k=1}^{\infty} \lambda_k a_k,
\end{equation}
and for any $\epsilon > 0,$
\begin{align*}
\sum_{k=1}^{\infty} |\lambda_k| < (1+\epsilon) \|f\|_{\mathfrak{B}_{r,\alpha}(\mathbb{Q}_p)}.
\end{align*}

Since $a_k$ is a central $(r, \alpha)$- block with the support $B^n,$ therefore we see that $ \mathcal{D}_\tau a_k$ is a central $(r, \alpha)$- block with the support $\tau^{-1}B^n$ and 
\begin{align*}
\|\mathcal{D}_\tau a_k\|_{L^r(\mathbb{Q}_p)} &= |\tau|_p^{-1/r} \|a_k\|_{L^r(\mathbb{Q}_p)} \\
& \leq |\tau|_p^{-1/r} |B^n|^{- \alpha} \\
& = |\tau|_p^{-1/r} |\tau|_p^{-\alpha} |\tau^{-1}B^n|^{- \alpha} \\
& = |\tau|_p^{-(1/r+\alpha)}|\tau^{-1}B^n|^{- \alpha}.
\end{align*}
From (\ref{eq2}), 
\begin{align*}
\mathcal{D}_\tau f & = \sum_{k=1}^{\infty} \lambda_k \mathcal{D}_\tau a_k \\
& = \sum_{k=1}^{\infty} |\tau|_p^{-(1/r+\alpha)}\lambda_k |\tau|_p^{(1/r+\alpha)} \mathcal{D}_\tau a_k \\
& = \sum_{k=1}^{\infty} \gamma_k b_k,
\end{align*}
where $\gamma_k = |\tau|_p^{-(1/r+\alpha)}\lambda_k$ and $b_k =|\tau|_p^{(1/r+\alpha)} \mathcal{D}_\tau a_k.$ By the definition of block spaces, we have $\mathcal{D}_\tau f \in  \mathfrak{B}_{r,\alpha}(\mathbb{Q}_p)$ and 
\begin{align*}
\|\mathcal{D}_\tau f\|_{\mathfrak{B}_{r,\alpha}(\mathbb{Q}_p)} & \leq \sum_{k=1}^{\infty} |\gamma_k| \\
& = \sum_{k=1}^{\infty}  |\tau|_p^{-(1/r+\alpha)} |\lambda_k| \\
& \leq |\tau|_p^{-(1/r+\alpha)} (1+\epsilon) \|f\|_{\mathfrak{B}_{r,\alpha}(\mathbb{Q}_p)}.
\end{align*}
As $\epsilon > 0$ was arbitrary, we obtain that
\begin{align*}
\|\mathcal{D}_\tau f\|_{\mathfrak{B}_{r,\alpha}(\mathbb{Q}_p)} \leq 2|\tau|_p^{-(1/r+\alpha)} \|f\|_{\mathfrak{B}_{r,\alpha}(\mathbb{Q}_p)}. 
\end{align*}
\end{proof}
The following $p$-adic Hardy-Hilbert type integral inequality is the main result of this paper.

\begin{theorem} \label{thm1}
Let $1 < r < \infty,~  \alpha > 0$ and let $p$-adic Hardy-Hilbert type integral operator $\mathscr{T}^p$ is defined by (\ref{eq3}). If $\mathcal{K}$ satisfies 
\begin{align} \label{eq4}
C_{r,\alpha} = 2(1-p^{-1})\sum_{k=-\infty}^{\infty}\mathcal{K}(1,p^{k}) p^{-k(1/r+\alpha-1)} < \infty.
\end{align}
Then
\begin{align*}
\|\mathscr{T}^pf\|_{\mathfrak{B}_{r,\alpha}(\mathbb{Q}^*_p)} \leq C_{r,\alpha} \|f\|_{\mathfrak{B}_{r,\alpha}(\mathbb{Q}^*_p)}, \\ \nonumber
\end{align*}
for all $f \in \mathfrak{B}_{r,\alpha}(\mathbb{Q}^*_p).$
\end{theorem}

\begin{proof}
Let $y=x\tau$ in (\ref{eq3}), then , by $dy=|x|_pd\tau$ we have
\begin{align} \nonumber
\mathscr{T}^pf(x) & = \int\limits_{\mathbb{Q}^*_p} \mathcal{K}(|x|_p, |x\tau|_p)f(\tau x)|x|_pd\tau \\ \nonumber
& =  \int\limits_{\mathbb{Q}^*_p} |x|_p^{-1}\mathcal{K}(1,|\tau|_p )f(\tau x)|x|_pd\tau \\ \nonumber
& =  \int\limits_{\mathbb{Q}^*_p}\mathcal{K}(1,|\tau|_p )f(\tau x)d\tau \\ \nonumber
& = \sum_{k=-\infty}^{\infty}~ \int\limits_{S^k} \mathcal{K}(1,|\tau|_p )\mathcal{D}_\tau f( x)d\tau \\ \label{min1}
& = \sum_{k=-\infty}^{\infty}\mathcal{K}(1,p^{k})~\int\limits_{S^k}\mathcal{D}_\tau f(x)d\tau.
\end{align} 
Let us first apply the norm $\|\cdot\|_{\mathfrak{B}_{r,\alpha}(\mathbb{Q}^*_p)}$ on both sides of \eqref{min1} and then using Theorem \ref{Min} and Theorem \ref{dil}, we get
\begin{align*}
\|\mathscr{T}^pf\|_{\mathfrak{B}_{r,\alpha}(\mathbb{Q}^*_p)} & \leq \sum_{k=-\infty}^{\infty}\mathcal{K}(1,p^{k})~\bigg\| \int\limits_{S^k}\mathcal{D}_\tau f(x)d\tau \bigg\|_{\mathfrak{B}_{r,\alpha}(\mathbb{Q}^*_p)}  \\
& \leq \sum_{k=-\infty}^{\infty}\mathcal{K}(1,p^{k})~\int\limits_{S^k}\|\mathcal{D}_\tau f\|_{\mathfrak{B}_{r,\alpha}(\mathbb{Q}^*_p)} d\tau \\
& \leq \sum_{k=-\infty}^{\infty}\mathcal{K}(1,p^{k})~\int\limits_{S^k}2|\tau|_p^{-(1/r+\alpha)} \|f\|_{\mathfrak{B}_{r,\alpha}(\mathbb{Q}^*_p)} d\tau \\
& = \sum_{k=-\infty}^{\infty}\mathcal{K}(1,p^{k}) 2 p^{-k(1/r+\alpha)} |S^k| \|f\|_{\mathfrak{B}_{r,\alpha}(\mathbb{Q}^*_p)} \\
& = \sum_{k=-\infty}^{\infty}\mathcal{K}(1,p^{k}) 2 p^{-k(1/r+\alpha)} p^{k}(1-p^{-1}) \|f\|_{\mathfrak{B}_{r,\alpha}(\mathbb{Q}^*_p)} \\
& = 2(1-p^{-1}) \sum_{k=-\infty}^{\infty}\mathcal{K}(1,p^{k}) p^{-k(1/r+\alpha-1)} \|f\|_{\mathfrak{B}_{r,\alpha}(\mathbb{Q}^*_p)} \\
& = C_{r, \alpha} \|f\|_{\mathfrak{B}_{r,\alpha}(\mathbb{Q}^*_p)}.\\
\end{align*}

Therefore (\ref{eq4}), assures that for all $f \in \mathfrak{B}_{r,\alpha}(\mathbb{Q}^*_p),$ we have
\begin{align*}
\|\mathscr{T}^pf\|_{\mathfrak{B}_{r,\alpha}(\mathbb{Q}^*_p)} \leq C_{r, \alpha} \|f\|_{\mathfrak{B}_{r,\alpha}(\mathbb{Q}^*_p)}.
\end{align*}
\end{proof}

As a consequence of Theorem \ref{thm1}, we obtain the $p$-adic Hilbert type inequality on $\mathfrak{B}_{r,\alpha}(\mathbb{Q}^*_p).$

\begin{theorem}
Let $1 < r < \infty,~ \alpha > 0$ and let $p$-adic Hilbert operator is defined by (\ref{Hilbert}). If $0< 1/r + \alpha < 1$ then there is a constant $C > 0$ such that for any $f \in \mathfrak{B}_{r,\alpha}(\mathbb{Q}^*_p)$  
\begin{align*}
\|H^pf\|_{\mathfrak{B}_{r,\alpha}(\mathbb{Q}^*_p)} \leq C \|f\|_{\mathfrak{B}_{r,\alpha}(\mathbb{Q}^*_p)}. 
\end{align*}
\end{theorem}

\begin{proof}
Let $\mathcal{K}(|x|_p,|y|_p) = \dfrac{1}{|x|_p~+~|y|_p}.$ Notice that, $\mathcal{K}$ is a nonnegative homogeneous function of degree $-1,$ and 
\begin{align*}
C_{r,\alpha} & = 2(1-p^{-1})\sum_{k=-\infty}^{\infty}\mathcal{K}(1,p^{k}) p^{-k(1/r+\alpha-1)} \\
& = 2(1-p^{-1})\bigg(\sum_{k=-\infty}^{0}\mathcal{K}(1,p^{k}) p^{-k(1/r+\alpha-1)} ~ + ~ \sum_{k=1}^{\infty}\mathcal{K}(1,p^{k}) p^{-k(1/r+\alpha-1)} \bigg) \\
& = 2(1-p^{-1})\bigg(\mathcal{K}(1,1)~+~ \sum_{k=-\infty}^{-1}\mathcal{K}(1,p^{k}) p^{-k(1/r+\alpha-1)} ~ + ~ \sum_{k=1}^{\infty}\mathcal{K}(1,p^{k}) p^{-k(1/r+\alpha-1)} \bigg) \\
& = 2(1-p^{-1})\bigg(\dfrac{1}{2}~+~ \sum_{k=1}^{\infty}\mathcal{K}(1,p^{-k}) p^{k(1/r+\alpha-1)} ~ + ~ \sum_{k=1}^{\infty}\mathcal{K}(1,p^{k}) p^{-k(1/r+\alpha-1)} \bigg) \\
& = 2(1-p^{-1})\bigg(\dfrac{1}{2}~+~ \sum_{k=1}^{\infty} \bigg[\dfrac{p^kp^{k(1/r+\alpha-1)}}{p^k+1} ~ + ~ \dfrac{p^{-k(1/r+\alpha-1)}}{1+p^{k}}\bigg] \bigg) \\
& = 2(1-p^{-1})\bigg(\dfrac{1}{2}~+~ \sum_{k=1}^{\infty} \dfrac{1}{1+p^k} ( p^{k(1/r+\alpha-1+1)}~ + ~ p^{-k(1/r+\alpha-1)}) \bigg),
\end{align*}
since we have $0< 1/r + \alpha < 1,$ which implies that $C_{r,\alpha} < \infty,$ and therefore Theorem \ref{thm1} gives the $p$-adic Hilbert type inequality on $\mathfrak{B}_{r,\alpha}(\mathbb{Q}^*_p).$
\end{proof}

We also obtain the $p$-adic Hardy type inequality on $\mathfrak{B}_{r,\alpha}(\mathbb{Q}^*_p).$

\begin{theorem}
Let $1 < r < \infty,~  \alpha > 0$ and let $p$-adic Hardy operator is defined by (\ref{Hardy}). If $ 0< 1/r + \alpha < 1$ there is a constant $C > 0$ such that for any $f \in \mathfrak{B}_{r,\alpha}(\mathbb{Q}^*_p)$  
\begin{align*}
\|\mathcal{H}^pf\|_{\mathfrak{B}_{r,\alpha}(\mathbb{Q}^*_p)} \leq C \|f\|_{\mathfrak{B}_{r,\alpha}(\mathbb{Q}^*_p)}. 
\end{align*}
\end{theorem}
\begin{proof}
Let $\mathcal{K}(|x|_p,|y|_p) = |x|_p^{-1} \Phi_E(|y|_p),$ where $\Phi_E$ is the characteristic function of $E = \{ y \in \mathbb{Q}_p : |y|_p \leq |x|_p\},$ and it is obviously a nonnegative homogeneous function of degree $-1.$ Moreover, we have 
\begin{align*}
C_{r,\alpha} & = 2(1-p^{-1})\sum_{k=-\infty}^{\infty}\mathcal{K}(1,p^{k}) p^{-k(1/r+\alpha-1)} \\
& = 2(1-p^{-1}) \sum_{k=- \infty}^{0} p^{-k(1/r+\alpha-1)}\\
& = 2(1-p^{-1}) \bigg( 1 +  \sum_{k=1}^{\infty} p^{k(1/r+\alpha-1)} \bigg) < \infty,~~~~~~~~~(\because 1/r + \alpha < 1 ).\\
\end{align*}
Hence, according to Theorem \ref{thm1}, we have $p$-adic Hardy type inequality on $\mathfrak{B}_{r,\alpha}(\mathbb{Q}^*_p).$
\end{proof}

\begin{theorem}
Let $1 < r < \infty,~ \alpha > 0$ and let the operator $\mathscr{D}^p$ is defined by (\ref{polya}). If $-\frac{\lambda}{2} < 1/r+\alpha < \frac{\lambda}{2}+1$ then there is a constant $C > 0$ such that for any $f \in \mathfrak{B}_{r,\alpha}(\mathbb{Q}^*_p)$  

\begin{align*}
\|\mathscr{D}^pf\|_{\mathfrak{B}_{r,\alpha}(\mathbb{Q}^*_p)} \leq C \|f\|_{\mathfrak{B}_{r,\alpha}(\mathbb{Q}^*_p)}. 
\end{align*}
\end{theorem}
\begin{proof}
Let

\begin{align} \label{kernel}
\mathcal{K}(|x|_p,|y|_p) = \dfrac{(|x|_p|y|_p)^{\frac{\lambda}{2}}}{\text{max}\{|x|_p,|y|_p\}^{\lambda+1}},~~~~~\lambda \geq 0.
\end{align}

We find that $\mathcal{K}$ is a nonnegative homogeneous function of degree $-1,$ and
 \begin{align*}
C_{r,\alpha} & = 2(1-p^{-1})\sum_{k=-\infty}^{\infty}\mathcal{K}(1,p^{k}) p^{-k(1/r+\alpha-1)} \\
& = 2(1-p^{-1})\bigg(\sum_{k=-\infty}^{0}\mathcal{K}(1,p^{k}) p^{-k(1/r+\alpha-1)} ~ + ~ \sum_{k=1}^{\infty}\mathcal{K}(1,p^{k}) p^{-k(1/r+\alpha-1)} \bigg) \\
& = 2(1-p^{-1})\bigg(\mathcal{K}(1,1)~+~ \sum_{k=1}^{\infty}\mathcal{K}(1,p^{-k}) p^{k(1/r+\alpha-1)} ~ + ~ \sum_{k=1}^{\infty}\mathcal{K}(1,p^{k}) p^{-k(1/r+\alpha-1)} \bigg) \\
& = 2(1-p^{-1})\bigg(1~+~ \sum_{k=1}^{\infty}p^{-\frac{k\lambda}{2}}p^{k(1/r+\alpha-1)} ~ + ~ \sum_{k=1}^{\infty}p^{-k(\frac{\lambda}{2}+1)} p^{-k(1/r+\alpha-1)} \bigg) \\
& = 2(1-p^{-1})\bigg(1~+~ \sum_{k=1}^{\infty} \bigg[ p^{k(1/r+\alpha-1-\frac{\lambda}{2})} ~ + ~ p^{-k(1/r+\alpha-1+\frac{\lambda}{2}+1)}\bigg] \bigg) < \infty,~~~~(\because -\frac{\lambda}{2} < 1/r+\alpha < \frac{\lambda}{2}+1). 
\end{align*}

Consequently, the boundedness of $\mathscr{D}^p$ on $\mathfrak{B}_{r,\alpha}(\mathbb{Q}^*_p)$ is assured by Theorem \ref{thm1}.
\end{proof}

\begin{remark} Taking $\lambda =0 $ in kernal (\ref{kernel}), we get the p-adic Hardy-Littlewood- P\'{o}lya inequality  on $\mathfrak{B}_{r,\alpha}(\mathbb{Q}^*_p)$ as follows:
\begin{align*}
\|\mathscr{P}^pf\|_{\mathfrak{B}_{r,\alpha}(\mathbb{Q}^*_p)} \leq C \|f\|_{\mathfrak{B}_{r,\alpha}(\mathbb{Q}^*_p)}. 
\end{align*}
\end{remark}


\bibliographystyle{amsplain}

\begin{thebibliography}{10}

\bibitem{AA} D. R. Adams and J. Xiao,\quotes{Morrey spaces in harmonic analysis,}Ark. Mat. \textbf{50}(2), 201--230 (2012).

\bibitem{bib1} S. Albeverio, A. Yu. Khrennikov, and V. M. Shelkovich,\quotes{Harmonic analysis in the $p$-adic Lizorkin spaces: fractional operators, pseudo-differential equations, $p$-adic wavelets, Tauberian theorems,} J. Fourier Anal. Appl. \textbf{12}(4), 393-425 (2006).

\bibitem{QJ1} S. Ashraf and Q. Jahan,\quotes{Dilation Operators in Besov Spaces over Local Fields,} Adv. Oper. Theory  \textbf{8}(2), (2023).

\bibitem{QJ2} S. Ashraf and Q. Jahan,\quotes{ A Note on Boundedness of Singular Integral Operators on Function Spaces over Local Fields,} preprint, (2022)

\bibitem{10} T. Batbold, Y. Sawano and G. Tumendemberel,\quotes{Sharp bounds for certain {$m$}-linear integral operators on {$p$}-adic function spaces,} Filomat \textbf{36}(3), 801--812 (2022).

\bibitem{BJ1} B. Behera and Q. Jahan,\quotes{ Wavelet packets and wavelet frame packets on local fields of positive characteristic,} J. Math. Anal. Appl. \textbf{395}(1), 1--14 (2012).

\bibitem{BJ2} B. Behera and Q. Jahan,\quotes{Characterization of wavelets and MRA wavelets on local fields of positive characteristic,} Collect. Math. \textbf{66}(1), 33--53 (2015).

\bibitem{11} A. Benyi and  C. Oh,\quotes{Best constants for certain multilinear integral operators,} J. Inequal. Appl. \textbf{2006}(1), 1--12 (2006).

\bibitem{1} O. Blasco, A. Ruiz and L. Vega,\quotes{Non interpolation in Morrey-Campanato and block spaces,} Ann. Scuola Norm. Sup. Pisa Cl. Sci.(4) \textbf{28}(1), 31--40 (1999).

\bibitem{CH4} N. M. Chuong and H. D. Hung,\quotes{Maximal functions and weighted norm inequalities on local fields,} Appl. Comput. Harmon. Anal. \textbf{29}(3), 272--286 (2010).

\bibitem{CH5} N. M. Chuong, Harmonic, wavelet and p-adic analysis (World Scientific Publishing Co. Pte. Ltd., Hackensack, NJ, 2007).


\bibitem{CH} N. M. Chuong and D. V. Duong,\quotes{The $p$-adic Hardy-Ces\'{a}ro operators on weighted Morrey-Herz space,} $p$
-Adic Numbers Ultrametric Anal. Appl. \textbf{8}(3), 204-–216 (2016).

\bibitem{CH2} N. M. Chuong, D. V. Duong and N. D. Duyet,\quotes{Two-weighted estimates for multilinear Hausdorff operators on the Morrey-Herz spaces,} Adv. Oper. Theory \textbf{5}(4), 1780-–1813 (2020).

\bibitem{CH3} N. M. Chuong, D. V. Duong and K. H. Dung,\quotes{Weighted Lebesgue and central Morrey estimates for $p$-adic multilinear Hausdorff operators and its commutators,} Ukr. Math. J. \textbf{73}(7), 1138-–1168 (2021).

\bibitem{CH6} Q. Y. Wu and Z. W. Fu,\quotes{Weighted {$p$}-adic {H}ardy operators and their commutators on {$p$}-adic central {M}orrey spaces,} Bull. Malays. Math. Sci. Soc. \textbf{40}(2), 635--654 (2017).


\bibitem{NM} N. M. Chuong and . N.M.. Duong,\quotes{ Weighted Hardy-Littlewood operators and commutators on p-adic functional spaces,} $p$-Adic Numbers Ultrametric Anal. Appl. \textbf{5}(1), 65--82 (2013).






\bibitem{5} K. H. Dung and D. V. Duong,\quotes{The $p$-adic Hausdorff Operator and Some Applications to Hardy–Hilbert Type Inequalities,} Russ. J. Math. Phys. \textbf{28}(3), 303--316 (2021).


\bibitem{9} K. H. Dung, D. V. Duong, and N. D. Duyet,\quotes{Weighted Triebel-Lizorkin and Herz Spaces Estimates for $p$-adic Hausdorff Type Operator and its Applications,} Anal. Math. \textbf{48}(3), 1--24 (2022).

\bibitem{ED} D. E. Edmunds and W. D. Evans, Hardy operators, function spaces and embeddings (Springer Monographs in Mathematics, Springer-Verlag, Berlin, 2004).

\bibitem{FU} Z. W. Fu, Q. Y. Wu and S. Z. Lu,\quotes{Sharp estimates of $p$-adic Hardy and Hardy-Littlewood-P\'olya operators,} Acta Math. Sin., Engl. Ser. \textbf{29}(1), 137--150 (2013).


\bibitem{new4} A. Hussain and N. Sarfraz,\quotes{The Hausdorff operator on weighted $p$-adic Morrey and Herz type spaces,} $p$-Adic Numbers Ultrametric Anal. Appl. \textbf{11}(2), 151--162 (2019).

\bibitem{HA} H. G. Hardy, E. J. Littlewood and G. Pólya, Inequalities (Cambridge university press, 1952).


\bibitem{3} K. P. Ho,\quotes{Calder\'{o}n-Zygmund Operators on Morrey and Hardy-Morrey Spaces in Locally Compact Vilenkin Groups,} $p$-Adic Numbers Ultrametric Anal. Appl.  \textbf{13}(3), 204--214 (2021).

\bibitem{HK3} K. P. Ho,\quotes{Hardy-Littlewood-p\'olya inequalities and hausdorff operators on block spaces,} Math. Inequal. Appl. \textbf{19}(2), 697--707 (2016).

\bibitem{IE} T. Izumi, E. Sato and K. Yabuta,\quotes{Remarks on a subspace of Morrey spaces,} Tokyo J. Math. \textbf{37}(1), 185--197 (2014).

\bibitem{KR} M. Krnic and J. Pecaric,\quotes{General Hilbert’s and Hardy’s inequalities,} Math. Inequal. Appl, \textbf{8}(1), 29--51 (2005). 


\bibitem{L1} S. Z. Lu and D. C. Yang,\quotes{The decomposition of Herz spaces on local fields and its applications,} J. Math. Anal. Appl. \textbf{196}(1), 296--313 (1995).


\bibitem{4} H. Li and J. Jin,\quotes{On a $p$-adic Hilbert-type integral operator and its applications,} J. Appl. Anal. Comput. \textbf{10}(4), 1326--1334 (2020).


\bibitem{MB} M. Molla and B. Behera,\quotes{Weighted norm inequalities for maximal operator of Fourier series,} Adv. Oper. Theory \textbf{7}(1), 1–-18 (2022).

\bibitem{MO} M. N. Molla,\quotes{$ H^ 1$, BMO, and John–Nirenberg Inequality on LCA Groups,} Mediterr. J. Math. \textbf{20}(2), 1--20 (2023).

\bibitem{SI} K. Phillips and M. Taibleson,\quotes{Singular integrals in several variables over a local field,} Pacific J. Math. \textbf{30}(1), 209-–231 (1969).

\bibitem{Kim} Y. C. Kim,\quotes{Carleson measures and the BMO space on the $p$-adic vector space,} Math. Nachr. \textbf{282}(9), 1278--1304 (2009).

\bibitem{R} K. S. Rim and J. Lee,\quotes{Estimates of weighted Hardy–Littlewood averages on the p-adic vector space,} J. Math. Anal. Appl.  \textbf{324}(2), 1470--1477 (2006).

\bibitem{AA1} D. R. Adams, Morrey spaces (Springer International Publishing, 2015).

\bibitem{book_impulse6} M. H. Taibleson, Fourier Analysis on Local Fields (Princeton University Press, 1975).

\bibitem{7} D. V. Duong and N. T. Hong,\quotes{Multilinear Hausdorff operator on $p$-adic functional spaces and its applications,} Anal. Math. Phys. \textbf{12}(3), 1--22 (2022).

\bibitem{new8} D. V. Duong and N. T. Hong,\quotes{Some new weighted estimates for $p$-adic multilinear Hausdorff type operator and its commutators on Morrey–Herz spaces,} Adv. Oper. Theory \textbf{7}(3), 1--21 (2022). 

\bibitem{BMO} S. S. Volosivets,\quotes{Hausdorff operators on $p$-adic linear spaces and their properties in Hardy, BMO, and Hölder spaces,} Math. Notes, \textbf{93}(3-4), 382--391 (2013).

\bibitem{new3} S. S. Volosivets,\quotes{Multidimensional Hausdorff operator on $p$-adic field,} $p$-Adic Numbers Ultrametric Anal. Appl. \textbf{2}(3), 252--259 (2010).

\bibitem{6} V. S. Vladimirov, I. V. Volovich and E. I. Zelenov, $p$-adic Analysis and Mathematical Physics(World Scientific, 1994).

\bibitem{YA} B. Yang, Hilbert-type integral inequalities(Bentham Science Publishers, 2010). 

\bibitem{YA1} B. Yang, The norm of operator and Hilbert-type inequalities(Science Press, Beijing, 2009).

\bibitem{YN} W. Yuan, W. Sickel and D. Yang, Morrey and Campanato Meet Besov, Lizorkin and Triebel(Springer-Verlag, Berlin, 2010).

\bibitem{Z} C. Zorko,\quotes{Morrey space,} Proc. Amer. Math. Soc. \textbf{98}(4), 586--592 (1986).

\end{thebibliography}

\end{document}